\documentclass[11pt]{amsart}
\usepackage{graphicx}

\usepackage{amsmath, amsfonts, amssymb, amsthm, hyperref, mathtools, mathrsfs}
\usepackage[table]{xcolor}
\usepackage{fullpage}

\newtheorem{theorem}{Theorem}[section]
\newtheorem{conjecture}[theorem]{Conjecture}

\newtheorem{proposition}[theorem]{Proposition}

\numberwithin{theorem}{section}
\numberwithin{equation}{section}

\def\l{\lambda}
\def\a{\alpha}

\def\N{\mathbb{N}}

\begin{document}

\title{Fixed perimeter analogues of some partition results}

\author{Gabriel Gray}
\address{University of Dayton}
\email{grayg1@udayton.edu}

\author{Emily Payne}
\address{University of Texas Rio Grande Valley}
\email{emily.payne01@utrgv.edu}

\author{Holly Swisher}
\address{Oregon State University}
\email{swisherh@oregonstate.edu}

\author{Ren Watson}
\address{University of Texas at Austin}
\email{renwatson@utexas.edu}

\thanks{The authors were supported by NSF grant DMS-2101906.}

\begin{abstract}
Euler's partition identity states that the number of partitions of $n$ into odd parts is equal to the number of partitions of $n$ into distinct parts.  Strikingly, Straub proved in 2016 that this identity also holds when counting partitions of any size with largest hook (perimeter) $n$. This has inspired further investigation of partition identities and inequalities in the fixed perimeter setting.  Here, we explore fixed perimeter analogues of some well-known partition results inspired by Euler's partition identity. 
\end{abstract}

\maketitle

\section{Introduction and Statement of Results}

A \textit{partition} of a positive integer $n$ is a nonincreasing sequence of positive integers called \textit{parts} that sum to $n$.  The partition function $p(n)$ counts the number of partitions of $n$, and $p(n\mid *)$ counts the number of partitions of $n$ satisfying some condition $*$.  One can visualize a partition $\pi$ with a \textit{Ferrers diagram}, in which each part $\pi_i$ is represented by a left-justified row of $\pi_i$ dots ordered from top to bottom.   

For any partition $\pi$, we let $\alpha(\pi)$ denote the largest part and $\lambda(\pi)$ the number of parts. We define the \emph{perimeter} of a partition $\pi$ to be $\alpha(\pi)+\lambda(\pi)-1$, which is the number of dots along the top row and left column of the Ferrers diagram of $\pi$, i.e., the largest hook length of $\pi$.  For example, the following is the Ferrers diagram of the partition $5+3+3$ which has perimeter $7$. 

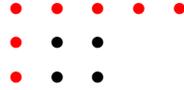
\begin{center}
    \begin{tabular}{c c c c c}
        \color{red} $\bullet$ & \color{red} $\bullet$ & \color{red} $\bullet$ & \color{red} $\bullet$ & \color{red} $\bullet$ \\
        \color{red} $\bullet$ & $\bullet$ & $\bullet$ &  &  \\
        \color{red} $\bullet$ & $\bullet$ & $\bullet$ &  &  
    \end{tabular} \\
\end{center}

As an analogue to the partition function $p(n)$, let $r(n)$ count the number of partitions with fixed perimeter $n$.  Thus $p(4)=5$, since 
\[
4, \enspace 3+1, \enspace 2+2, \enspace 2+1+1, \enspace 1+1+1+1 
\]
are the partitions of $4$, but $r(4)=8$, since 
\[
4, \enspace 3+1, \enspace 3+2, \enspace 3+3, \enspace 2+1+1, \enspace 2+2+1, \enspace 2+2+1, \enspace 1+1+1+1 
\]
are the partitions (of any $n$) that have perimeter $4$.

Partition identities, which equate partition counting functions with different conditions, have been of interest in this area for centuries.  One of the foundational partition identities, due to Euler, states that the number of partitions of $n$ into odd parts equals the number of partitions of $n$ into distinct parts.  Namely,
\begin{equation} \label{Euler}
p(n\mid \text{odd parts}) = p(n\mid \text{distinct parts}).
\end{equation}

In 2016, Straub \cite{Straub} proved the following beautiful analogue of Euler's identity for fixed perimeter partitions, 
\begin{equation} \label{thm:Straub}
r(n\mid \text{odd parts}) = r(n\mid \text{distinct parts}).
\end{equation}
In light of \eqref{thm:Straub}, there has been significant interest in investigating the relationship between classical partition identities and inequalities, and their counterparts in the fixed perimeter setting. 


\subsection{Franklin-type identities}
One direction of generalization of \eqref{Euler} is through its interpretation as
\[
O_{0,2}(n) = p(n\mid\text{no part divisible by }2) = p(n\mid\text{no part appears }\geq 2 \text{ times}) = D_{0,2}(n).
\]
Letting  $O_{j,k}(n)$ count the number of partitions of $n$ with exactly $j$ part sizes divisible by $k$ (the parts can repeat), and $D_{j,k}(n)$ count the number of partitions of $n$ with exactly $j$ part sizes with parts appearing at least $k$ times, Franklin, \cite{Franklin} generalizing a result of Glaisher \cite{glaisher1883}, proved that for all $k \geq 2$, $j\geq 0$, $n\geq 1$,
\begin{equation} \label{Franklin}
O_{j,k}(n) = D_{j,k}(n).
\end{equation}

Two recent variations on Franklin's identity are given as follows.  In the case when $j=1$, Amdeberhan, Andrews, and Ballantine \cite{AAB} proved that if $O_{1,k}^u(n)$ is the number of partitions counted by $O_{1,k}(n)$ where the unique part divisible by $k$ is occurs exactly $u$ times, and  $D_{1,k}^u(n)$ is the number of partitions counted by $D_{1,k}(n)$ where $u$ is the unique part that is repeated at least $k$ times, then
\begin{equation} \label{AAR1}
O_{1,k}^u(n) = D_{1,k}^u(n).
\end{equation}
Furthermore, if we define $O_{j,k,b}(n)$ to be the number of partitions of $n$ with exactly $j$ part sizes divisible by $kb$ (parts can repeat) and $D_{j,k,b}(n)$ to be the number of partitions of $n$ with exactly $j$ part sizes both divisible by $b$ and with parts appearing at least $k$ times, then Gray et al. \cite{GHKPSW} proved that for all $k \geq 2$, $j\geq 0$, $b\geq 1$, $n\geq 1$,
\begin{equation} \label{REU1thm}
O_{j,k,b}(n) = D_{j,k,b}(n).
\end{equation}

Define fixed perimeter analogues of $O_{j,k}(n)$ and $D_{j,k}(n)$ by
\begin{align*}
FO_{j,k}(n) &= r(n\mid\text{exactly }j \text{ parts are divisible by }k), \\ 
FD_{j,k}(n) &= r(n\mid\text{exactly }j \text{ parts appear }\geq k \text{ times}).
\end{align*}
In 2022, Andrews, Amdeberhan, and Ballantine \cite[Thm. 1.4]{AAB} proved that for all $n\geq 1$,
\begin{equation}\label{AAB12}
FO_{1,2}(n) = FD_{1,2}(n),
\end{equation}
which is a direct fixed perimeter analogue of the $j=1$, $k=2$ case of \eqref{Franklin}.  Furthermore in 2023, Lin, Xiong, and Yan \cite{LXY}, proved a fixed perimeter identity related to the $k=2$ case of Franklin's identity \eqref{Franklin}.  Namely, they show\footnote{We have reworded the statement of \cite[Thm. 3]{LXY} to more easily compare with related results discussed here.} that for any $j\geq 0$, $n\geq 1$,
\begin{equation}\label{LXY}
r(n \mid \text{number of even parts} =j ) = r(n \mid \text{(number of parts) - (number of distinct parts)} =j ).
\end{equation}
Here the left hand side is a direct fixed perimeter analogue of $O_{j,2}(n)$, while the right hand side is not a direct analogue of $D_{0,2}(n)$.

Our first result is the following direct fixed perimeter analogue of the $k=2$ case of Franklin's identity, generalizing \eqref{thm:Straub} and \eqref{AAB12}. 

\begin{theorem} \label{k=2}
For all $n$, $j \geq 0$, $$FO_{j,2}(n)=FD_{j,2}(n).$$
\end{theorem}
In light of \eqref{LXY} we also have the following corollary,
\[
r(n \mid \text{exactly $j$ parts are repeated}) = r(n \mid \text{(number of parts) - (number of distinct parts)} =j ).
\]
For $k \geq 2$, $FO_{j,k}(n)$ and $FD_{j,k}(n)$ are not equal for all $n$, however computational evidence suggests the following.

\begin{conjecture}\label{fofdconjecture}
If $j \geq 0, k \geq 2$, then $FD_{j,k}(n) \geq FO_{j,k}(n)$ for sufficiently large $n$.
\end{conjecture}

\subsection{Alder-type identities and inequalities}
A second direction of generalization of Euler's identity \eqref{Euler} is through its interpretation as
\[
Q_1^{(1)}(n) = p(n\mid\text{parts} \equiv \pm 1 \!\!\! \pmod{4}) = p(n\mid\text{parts are } 1 \text{-distinct}) = q_1^{(1)}(n),
\]
where we say a partition has $d$-distinct parts if the difference between any two parts is at least $d$.  From this perspective, the first Rogers-Ramanujan identities may be viewed as extensions.  In terms of partitions, they state that 
\begin{align*}
Q_2^{(1)}(n) = p(n\mid\text{parts} \equiv \pm 1 \!\!\! \pmod{5}) & = p(n\mid\text{parts are } 2 \text{-distinct}) = q_2^{(1)}(n), \\
Q_2^{(2)}(n) = p(n\mid\text{parts} \equiv \pm 2 \!\!\! \pmod{5}) & = p(n\mid\text{parts are } 2 \text{-distinct and} \geq 2) = q_2^{(2)}(n).
\end{align*}
Although in general
\[
Q_d^{(a)}(n) = p(n\mid\text{parts} \equiv \pm a \!\!\! \pmod{d+3}) \neq p(n\mid\text{parts are $(d+3)$-distinct and} \geq a) = q_d^{(a)}(n),
\]
the Alder-Andrews theorem \cite{AndrewsAlder, Yee, Yee2, AJLO} establishes that for all $d,n\geq 1$,
\begin{equation}\label{Alder}
q_d^{(1)}(n) \geq Q_d^{(1)}(n),
\end{equation}
and the problem for higher $a$ has been studied as well by Kang, Park and others \cite{KangPark, REU20, KangKim, SS, REU22, IT}.  In particular, Cho, Kang, and Kim \cite{CKK} have recently shown that for sufficiently large $d$,
\begin{equation}\label{CKK}
q_d^{(a)}(n) \geq Q_d^{(a)}(n)
\end{equation}
holds for all but a few values of $n$.

In the fixed perimeter setting, Fu and Tang \cite{FuTang2} show that for all $d,n \geq 1$, 
\begin{equation} \label{FuTangFixed}
h_d(n)= r(n \mid \text{parts are $d$-distinct}) = r(n \mid \text{parts are }\equiv 1 \!\!\! \pmod{d+1}) = f_d(n),
\end{equation}
which reduces to \eqref{Euler} when $d=1$.  They further established a refinement of their result when $d=1$, namely that the number of partitions counted by $h_1(n)$ with exactly $k$ parts is equal to the number of partitions counted by $f_1(n)$ with largest part of size $2k-1$ with both enumerated by $\binom{n-k}{k-1}$.  Waldron \cite{Waldron} also gave a different refinement of \eqref{FuTangFixed}.  Chen, Hernandez, Shields, and the third author \cite{CHSS} generalized \eqref{FuTangFixed} as follows.  Define for $d,n \geq 1$ and $1 \leq a \leq d+1$,
\begin{align*} \label{def:hda(n)=fda(n)}
h_{d}^{(a)}(n) &=  r(n\mid\text{parts are $d$-distinct and}\geq a), \\
f_{d}^{(a)}(n) &=  r(n\mid\text{parts are}\equiv a \!\!\! \pmod{d+1}).
\end{align*}
Then for positive integers $d$, $n$, $1 \leq a \leq d+1$, Chen et al. \cite{CHSS} proved that 
\begin{equation} \label{aldertype}
h_d^{(a)}(n)=f_d^{(a)}(n),
\end{equation} 
and moreover that the number of partitions counted by $h_d^{(a)}(n)$ with $\l$ parts equals the number of partitions counted by $f_d^{(a)}(n)$ with largest part $a+(d+1)(\l-1)$.  Chen et al. \cite{CHSS} also proved that if  $\ell_d^{(a)}(n)= r(n\mid \text{parts are} \equiv \pm a \!\! \pmod{d+3})$, then for all $d,n \geq 1$, and $a < \frac{d+3}{2}$, 
\begin{equation} \label{revAlder}
\ell_d^{(a)}(n) \geq h_d^{(a)}(n),
\end{equation}
thus giving a type of reverse analogue of \eqref{CKK}.

\subsection{Beck-type companion identities} 
A third direction of study evolving from Euler's identity \eqref{Euler} was initiated by Beck \cite{beck2017oeis} who conjectured that for all $n\geq 1$,
\begin{equation} \label{Beck}
\sum_{\pi \in \mathcal{O}_{0,2}(n)} \l(\pi) - \sum_{\pi \in \mathcal{D}_{0,2}(n)} \l(\pi) = O_{1,2}(n) = D_{1,2}(n),
\end{equation}
where $\mathcal{O}_{0,2}(n)$, $\mathcal{D}_{0,2}(n)$, are the sets of partitions counted by $O_{0,2}(n)$ and $D_{0,2}(n)$, respectively, and $\l(\pi)$ counts the total number of parts in $\pi$.  Equation \eqref{Beck} was proved by Andrews \cite{AndrewsBeck}.  Observe that \eqref{Beck} is counting the excess in the total number of parts in partitions into odd parts over the total number of parts in partitions into distinct parts.  

Generally, given two partition counting functions, $a(n)=p(n\mid *)$ and $b(n)=p(n\mid \heartsuit)$, with $A(n)$ and $B(n)$ the sets of partitions counted by $a(n)$ and $b(n)$, respectively, define 
\begin{equation}\label{def:E(a,b)}
E(a(n),b(n)) = \sum_{\pi \in A(n)} \!\!\! \l(\pi) - \!\!\! \sum_{\pi \in B(n)} \!\!\! \l(\pi),
\end{equation}
so that $E(a(n),b(n))$ counts excess in the total number of parts in all partitions counted by $a(n)$ over the total number of parts in all partitions counted by $b(n)$.  Given a partition identity $a(n)=b(n)$, an identity which counts the excess in number of parts (perhaps with certain conditions) for $a(n)$ over $b(n)$ is now called a Beck-type companion identity.  Beck-type identities have been established by Ballantine and several other researchers, see for example \cite{FuTang1, AndrewsBallantine, BallantineWelch2, BallantineWelch, BallantineFolsom}.

In the fixed perimeter setting, Andrews, Amdeberhan, and Ballantine \cite{AAB} prove a Beck-type companion identity for \eqref{thm:Straub}, which is a fixer-perimeter analogue of \eqref{Beck}.  They show in \cite[Cor. 1.8]{AAB} that
\[
E(FO_{0,2}(n), FD_{0,2}(n)) = FO_{1,2}(n \mid \text{parts} \geq 2).
\]

Let $\mathcal{F}_d^{(a)}(n)$ be the set of partitions counted by $f_d^{(a)}(n)$ so that $f_d^{(a)}(n) = |\mathcal{F}_d^{(a)}(n)|$.  Define $fp_d^{(a,b)}(n)$ to be the number of ordered pairs of partitions $(\pi_1,\pi_2)$ with $\pi_1\in \mathcal{F}_d^{(a)}(n-m)$ and $\pi_2\in \mathcal{F}_d^{(b)}(m)$, for some $1\leq m \leq n-1$.  Then 
\begin{equation} \label{fpdab}
fp_d^{(a,b)}(n) = \left|\bigcup_{m=1}^{n-1} \left(\mathcal{F}_d^{(a)}(n-m) \times \mathcal{F}_d^{(b)}(m)\right)\right| = \sum_{m=1}^{n-1} f_d^{(a)}(n-m) f_d^{(b)}(m).
\end{equation}

Our next result is the following Beck-type companion identity for \eqref{aldertype}.

\begin{theorem} \label{beck}
For fixed positive integers $d$, $n$, and $1 \leq a \leq d+1$,
\[
E(f_d^{(a)}(n), h_d^{(a)}(n)) = fp_{d+1}^{(a,1)}(n) - fp_{d+1}^{(a,d-1)}(n).
\]
\end{theorem}

We note that from \eqref{fpdab}, it follows using \eqref{aldertype} (or using Theorem \ref{perimST}) that $fp_{d+1}^{(a,1)}(n) \geq fp_{d+1}^{(a,d-1)}(n)$ for all $n\geq 1$.   We also note that a different interpretation of $E(f_d^{(a)}(n), h_d^{(a)}(n))$ is described in \cite{REU24Proc}.

\subsection{S-T type inequalities}
Andrews \cite[Thm. 3]{AndrewsAlder} proved the following result in his study of \eqref{Alder}.  Namely, if $S = \{a_0, a_1, a_2, \ldots \}$ and $T = \{b_0, b_1, b_2, \ldots\}$ have strictly increasing positive integer entries with $b_0=1$ and $a_i \geq b_i$ for all $i$, then 
\begin{equation} \label{ST}
    p(n \mid \text{parts in } T) \geq p(n \mid \text{parts in } S).
\end{equation}
Using an injective argument due to Yee \cite{Yee2}, \eqref{ST} has been extended first by Kang and Park \cite{KangPark} and then by Duncan et al. \cite{REU20}.  In particular, if we replace the condition that $b_0=1$ with $b_0=m$ and further require that $m$ divides each $b_i$, then we can again conclude \eqref{ST}. 

We find in our next result that the fixed perimeter analogue of \eqref{ST} holds true without needing $b_0=1$ or extra divisibility of $b_i$.  
\begin{theorem} \label{perimST}
Let $S = \{a_0, a_1, a_2, \ldots \}$ and $T = \{b_0, b_1, b_2, \ldots\}$ have strictly increasing positive integer entries with $a_i \geq b_i$ for all $i \geq 0$. Then 
\[
r_S(n)\leq r_T(n).
\]
\end{theorem}

As with \eqref{ST}, we expect Theorem \ref{perimST} to have a variety of applications for fixed perimeter partitions.  For example, Theorem \ref{perimST} provides a quick proof of \eqref{revAlder} due to the equality in \eqref{aldertype}.  And Theorem \ref{perimST} gives an alternative proof to the shift inequalities $\ell_{d+1}^{(a)}(n) \leq \ell_d^{(a)}(n)$, $h_{d+1}^{(a)}(n) \leq h_d^{(a)}(n)$, and $h_d^{(a+1)}(n) \leq h_d^{(a)}(n)$ given in \cite{CHSS} as well.

\subsection{Kang-Kim type asymptotics}
Define
\[
Q_m^{(m_1,m_2)}(n) = p(n\mid\text{parts} \equiv m_1,m_2 \!\!\! \pmod{m}). 
\]
Note that the modulus has shifted from $m+3$ to $m$ in this notation from the $Q_m^{(a)}(n)$ notation.  Kang and Kim \cite{KangKim} have established that there is a specific tipping point in the asymptotics of $q_d^{(a)}(n)$ versus $Q_m^{(m_1,m_2)}(n)$ depending on $m$ and $d$.  They show that for integers $0 \leq m_1<m_2<m$ and $a, d \geq 1$,
\begin{align} \label{kang-kim}
\lim _{n \rightarrow \infty}\left(q_d^{(a)}(n)-Q_m^{\left(m_1, m_2\right)}(n)\right) &= +\infty  \; \text { if } m > \left\lfloor\frac{\pi^2}{3 A_d}\right\rfloor, \\
\lim _{n \rightarrow \infty}\left(q_d^{(a)}(n)-Q_m^{\left(m_1, m_2\right)}(n)\right) &= -\infty \; \text { if } m \leq \left\lfloor\frac{\pi^2}{3 A_d}\right\rfloor, \nonumber
\end{align}
where $A_d = \frac{d}{2} \log ^2 \alpha_d+\sum_{r =1}^\infty  r^{-2}\alpha_d^{r d}$ for $\alpha_d$ the unique real root of $x^d+x-1$ in $(0,1)$.

We are interested in investigating this from the fixed perimeter perspective.  For $0 < a < b \leq d$, define
\[
\ell_d^{(a,b)}(n) = r(n\mid \text{parts are} \equiv a, b \!\!\! \pmod{d}).
\]
where again note the modulus has shifted from $d+3$ to $d$ in this notation to allow for cleaner statements and proofs.  

To consider an analogue for \eqref{kang-kim}, we compare $h_d^{(a)}(n)$ and $\ell_m^{(m_1,m_2)}(n)$.  Observe that by \eqref{aldertype}, 
\[
h_d^{(a)}(n) = f_d^{(a)}(n) = \ell_{2(d+1)}^{(a, a+d+1)}(n).
\]
Thus, our problem reduces to comparing shifted $\ell_d^{(a,b)}(n)$.   In \eqref{ellabcount} we establish a counting function for $\ell_d^{(a,b)}(n)$.  We make the following conjecture based on numerical evidence.

\begin{conjecture}\label{kangkimconjecture}
For positive integers $0 < m_1 < m_2 \leq m$, $0 < a \leq d$,
\begin{align*}
	\lim_{n\to \infty}\left(h_d^{(a)}(n)-\ell_m^{(m_1,m_2)}(n)\right) = \begin{cases}
	    +\infty & m > 2d+2,\\
            -\infty & m < 2d+2.
	\end{cases}
\end{align*}
\end{conjecture}

When $m=2d+2$ the situation is more complicated.  However, we can establish what happens in certain cases by the following proposition. 

\begin{proposition} \label{FPKKcase}
Let $0 < a_1, a_2, b_1, b_2 \leq d$ be positive integers such that $a_1\leq a_2$ and $b_1\leq b_2$ are not both equalities.  Then for all $n\geq 1$, 
\[
\lim_{n\to \infty}\left(\ell_d^{(a_1, b_1)}(n)-\ell_d^{(a_2,b_2)}(n)\right) = +\infty.
\]
Moreover, $\ell_d^{(a_1,b_1)}(n) \geq \ell_d^{(a_2,b_2)}(n)$ for all $n\geq 1$.
\end{proposition}

\noindent From Proposition \ref{FPKKcase} it follows that for positive integers $0 < m_1 < m_2 \leq 2d+2$ and $1 \leq a \leq d+1$,
\[
\lim_{n\to \infty}\left(h_d^{(a)}(n)-\ell_{2d+2}^{(m_1,m_2)}(n)\right) = 
\begin{cases}
0 & m_1 = a, \, m_2 = a+d+1, \\
+\infty & m_1 \geq a, \, m_2 \geq a+d+1 \text{ not both equal},\\
-\infty & m_1 \leq a, \, m_2 \leq a+d+1 \text{ not both equal}.
\end{cases}
\]
Outside of these cases, what happens in the limit when $m=2d+2$ is unclear from our experimentation. 

We now outline the remainder of this paper.  In Section \ref{sec:genfn}, we establish generating functions for $FD_{j,2}(n)$, $FO_{j,2}(n)$, and $\ell_d^{(a,b)}(n)$, as well as prove Theorem \ref{k=2} both via generating functions and with a combinatorial bijection.  In Section \ref{sec:perimST} we give two proofs of Theorem \ref{perimST}, one using counting functions and one via a set injection.  In Section \ref{sec:Beck} we prove Theorem \ref{beck}.   In Section \ref{sec:perim-kangkim}, we prove Proposition \ref{FPKKcase} and conclude with an alternative formula for studying $\ell_d^{(a,b)}(n)$.

\section{Fixed Perimeter Generating Functions and Proofs of Theorem \ref{k=2}} \label{sec:genfn}

For any partition $\pi$ of perimeter $n$, its largest part $\alpha=\alpha(\pi)$ (or colloquially, its \emph{arm length}), its number of parts $\lambda=\lambda(\pi)$ (its \emph{leg length}), and perimeter $n$ are related by 
$$n=\alpha+\lambda-1.$$

Fu and Tang \cite{FuTang2} associates to each partition $\pi$ a unique word in $\{E,N\}$, called the \emph{profile} of $\pi$, corresponding to its Ferrers diagram by tracing east and north steps from the lower left corner to the upper right corner.  For each step, an east move is denoted $E$ and north is denoted $N$.  For example, the partition $2+2+1$  
\[
\begin{matrix}
        \bullet  & \bullet \\
        \bullet  & \bullet \\
        \bullet & 
\end{matrix}
\]
is interpreted as the word $ENENN$.  The profile for any partition $\pi$ must start with an $E$ and end with an $N$.  Moreover, the number of $E$s in the profile gives the size of the largest part of $\pi$ and the number of $N$s gives the number of parts in $\pi$.  Thus the perimeter of $\pi$ is one less than the length of the profile.  Fu and Tang \cite[Thm. 2.3]{FuTang2} observe that a three-variable generating function for $$r(\a,\l, n) := r(n \mid \text{largest part } \a \text{ and } \l \text{ parts})$$ is 
\begin{equation}\label{FT2.1}
\sum_{n,\a,\l\geq 1} r(\a,\l, n) x^\a y^\l q^n = \frac{xyq}{1-(xq+yq)}.
\end{equation}
This follows because $x$, $y$ are tracking the size of the largest part and number of parts, respectively, so each $E$, $N$ will contribute an $x$ and $y$, respectively.  Further, $q$ is keeping track of the perimeter $\alpha + \lambda -1$ (note $r(\a,\l, n)=0$ otherwise).  Since every partition's word must start with an $E$ and end with an $N$, which together will account for only one square of the perimeter, the generating function will have an $xyq$ in the numerator.  For the rest of the perimeter, each $E$ and $N$ will each contribute one $q$.  Together, we have that the first and last $E,N$ contribute $xyq$, whereas each intermediate $E$ contributes a $xq$, and each intermediate $N$ contributes a $yq$.  This explains \eqref{FT2.1}, and as an immediate corollary, setting $x=y=1$ gives $\sum_{n\geq 1} r(n) q^n = q/(1-2q)$, so $r(n)=2^{n-1}$.  This method can be useful for constructing generating functions for other fixed-perimeter partition counting functions, which we do here.  For convenience we write $E_m$ to represent $E\cdots E$ with $m$ copies of $E$, and $N_m$ to represent $m$ copies of $N$.

We now prove Theorem \ref{k=2} by establishing generating functions for $FO_{j,2}(n)$ and $FD_{j,2}(n)$.
\begin{proof}[Proof of Theorem \ref{k=2} via generating functions]
Defining
\begin{align*}
A &= \frac{xq}{1-yq}, \\
B &= xq\left(1+\frac{zyq}{1-yq}\right) = xq\left( \frac{1-(1-z)yq}{1-yq} \right), \\
C &= xq\left(1+ yq + \frac{zy^2q^2}{1-yq}\right) = xq\left( \frac{1-(1-z)y^2q^2}{1-yq} \right),
\end{align*}
we observe that $A$ generates one choice of term $EN_m$, $B$ generates one $EN_m$ with $z$ tracking  when $m\geq 1$, and $C$ generates one $EN_m$, with $z$ tracking when $m\geq 2$.  We also note that the first $E$ and last $N$ of a profile together only add $1$ to the perimeter.  Thus $xyq$ generates this pair (or another single pair of $E$ and $N$ in the profile if more convenient).

Observe that for a partition $\pi$, we can write $$\pi = 1^{m_1} 2^{m_2} \cdots \alpha^{m_\alpha +1}$$ to denote that the multiplicity of the largest part $\alpha$ is $m_\a +1$, where $m_\a \geq 0$ and the multiplicity of each smaller part $i<\a$ in $\pi$ is $m_i\geq 0$.  From this perspective we see that the profile of $\pi$ is 
\begin{equation}\label{profile1}
EN_{m_1}EN_{m_2} \cdots EN_{m_\a}N.
\end{equation}

Let
\[
FD_{j,2}(\alpha,\lambda, n) =  r(n \mid \text{largest part } \a, \l \text{ total parts, and exactly } j \text{ repeated parts}).
\]
So a profile for a partition $\pi$ counted by $FD_{j,k}(\alpha,\lambda, n)$ has the form \eqref{profile1} where exactly $j$ indices $i$ satisfy $m_i\geq 2$.  Thus, $$\frac{1}{1-C}$$ generates $EN_{m_1}EN_{m_2} \cdots EN_{m_{\a-1}}$ of any length (could be empty).  So it remains to generate the largest part of $\pi$, represented by $EN_{m_\a}N$.  As described above, we can generate the pair of final $E$ and final $N$ by $xyq$, so that leaves only the final $N_{m_\a}$.  Since the largest part $\a$ is repeated if and only if $m_\a \geq 1$, we generate $N_{m_\a}$ by $B/xq$ so that $z$ accurately tracks when $\a$ is repeated.  Putting this together, we have that
\begin{equation}\label{FDxyq}
\sum_{\a,\l,n\geq 1} \sum_{j\geq 0} FD_{j,2}(\alpha,\lambda, n)x^\a y^\l z^j q^n = \frac{yB}{1-C} = \frac{xyq(1-(1-z)yq)}{1-(x+y)q+(1-z)xy^2q^3}.
\end{equation}

Similarly, let 
\[
FO_{j,2}(\alpha,\lambda, n) =  r(n \mid \text{largest part } \a, \l \text{ total parts, and exactly } j \text{ even parts}).
\]
So a profile for a partition $\pi$ counted by $FO_{j,k}(\alpha,\lambda, n)$ has the form \eqref{profile1} where exactly $j$ even indices $2i$ satisfy $m_{2i}\geq 1$.  While $1/(1-A)$ generates any number of blocks $EN_{m_i}$ with no $z$ and $1/(1-B)$ generates any number of blocks $EN_{m_i}$ with $z$ tracking when $m\geq 1$, we generate these in pairs.  Observe that $$\frac{1}{1-AB}$$ generates equal numbers of terms $EN_{m_i}$ (could be empty) with exactly half tracking when $m\geq 1$.  Thus we break the profile in \eqref{profile1} into the first block $EN_{m_1}$ with the final $N$, which is generated by $yA$ since $1$ is odd, then the pairs $EN_{m_2} \cdots EN_{m_{2 \lceil \frac{\a}{2} \rceil -1}}$, which are generated by $1/(1-AB)$, and then to account for the possibility that the largest part is even, we generate $EN_{m_\a}$ (or nothing if not) by $(1+ zA)$ since if $\alpha$ is even we need to track it with $z$.   Putting this together, we have that
\begin{equation}\label{FOxyq}
\sum_{\a,\l,n\geq 1} \sum_{j\geq 0} FO_{j,2}(\alpha,\lambda, n)x^\a y^\l z^j q^n = \frac{yA(1+ zA)}{1-AB} = \frac{xyq(1-(y-xz)q)}{1-2yq + (y^2-x^2)q^2 + (1-z)x^2yq^3}.
\end{equation}

Setting $x=y=1$ in \eqref{FDxyq} and \eqref{FOxyq} gives that 
\[
\sum_{n\geq 1} \sum_{j\geq 0} FD_{j,2}(n) z^jq^n = \frac{q(1-(1-z)q)}{1-2q+(1-z)q^3} = \sum_{n\geq 1} \sum_{j\geq 0} FO_{j,2}(n) z^jq^n,
\]
which yields the desired equality in Theorem \ref{k=2}. 
\end{proof}

We next give a combinatorial proof of Theorem \ref{k=2} by constructing a bijection between the sets of partitions counted by $FO_{j,2}(n)$ and $FD_{j,2}(n)$. 

\begin{proof}[Proof of Theorem \ref{k=2} via combinatorial bijection]
For fixed $j\geq 0$ and $n\geq 1$, let $\mathcal{FO}_{j,2}(n)$ and $\mathcal{FD}_{j,2}(n)$ denote the sets of partitions counted by $FO_{j,2}(n)$ and $FD_{j,2}(n)$, respectively.  As in \eqref{profile1}, for a partition $\pi\in \mathcal{FD}_{j,2}(n)$, write the profile of $\pi$ as
\[
EN_{m_1}EN_{m_2} \cdots EN_{m_\a}N,
\]
where each $m_i\geq 0$ is the multiplicity of the part $i$ in $\pi$.
Define $\varphi: \mathcal{FD}_{j,2}(n) \rightarrow \mathcal{FO}_{j,2}(n)$ by making the following changes to the profile of $\pi$.  First, map the initial $E\mapsto E$.  Then for each $1\leq i \leq \a-1$, map 
\[
N_{m_i}E \longmapsto \begin{cases} N &\text{ if } m_i=0, \\  EN_{m_i-1}E &\text{ if } m_i\geq 1. \end{cases}
\]
And finally, map the final
\[
N_{m_\a}N \longmapsto \begin{cases} N &\text{ if } m_\a=0, \\  EN_{m_\a} &\text{ if } m_\a\geq 1. \end{cases}
\]
Observe that the length of the profile remains unchanged, which means the perimeter of $\varphi(\pi)$ is $n$.  Furthermore, since $N_{m_i}E \mapsto N$ when $m_i=0$ and $N_{m_i}E \mapsto EE$ when $m_i=1$, the only way to obtain an even part in $\varphi(\pi)$ less than the largest part $\a$, is when $m_i\geq 2$ for $1\leq i \leq \a-1$ (and then $EN_{m_i-1}E$ makes the number of $E$s odd again).  So if the largest part $\a$ is distinct ($m_\a=0$), then the final $N\mapsto N$ ensures that the largest part of $\varphi(\pi)$ is odd, while if the largest part $\a$ repeats ($m_\a\geq 1$), then the final $N_{m_\a}N\mapsto EN_{m_\a}$ makes the largest part of $\varphi(\pi)$ even.  Thus since for $\pi\in \mathcal{FD}_{j,2}(n)$ there are exactly $j$ indices $i$ for which $m_i\geq 2$, we see that the number of even parts in $\varphi(\pi)$ is exactly $j$, so $\varphi(\pi)\in \mathcal{FO}_{j,2}(n)$.

To define an inverse map $\psi: \mathcal{FO}_{j,2}(n) \rightarrow \mathcal{FD}_{j,2}(n)$, we write the profile of $\pi\in \mathcal{FO}_{j,2}(n)$ as
\[
EN_{m_1}EN_{m_2} \cdots EN_{m_{(2\lceil\frac{\a}{2}\rceil -1)}}  \, \star,
\]
where
\[
\star = \begin{cases} N &\text{if } \a \text{ odd}, \\  EN_{m_\a}N &\text{if } \a \text{ even},  \end{cases}
\]
and again each $m_i\geq 0$ is the multiplicity of the part $i$ in $\pi$.
Define $\psi: \mathcal{FO}_{j,2}(n) \rightarrow \mathcal{FD}_{j,2}(n)$ by changing the profile of $\pi$ as follows.  For each $1\leq i \leq \a$, map 
\[
EN_{m_i} \longmapsto \begin{cases} E_{m_i+1} &\text{ if } i \text{ odd}, \\  N_{m_i+1} &\text{ if } i \text{ even}, \end{cases}
\]
and map the final $N\mapsto N$.  Observe that the length of the profile remains unchanged, which means the perimeter of $\psi(\pi)$ is $n$.  For $\pi\in \mathcal{FO}_{j,2}(n)$ there are exactly $j$ even indices $2i\leq \a$ that satisfy $m_{2i}\geq 1$ when $2i<\a$ ($\a$ is counted when $\a$ is even).  The part $EN_{m_{2i-1}}EN_{m_{2i}}$ in the profile of $\pi$ corresponds to $m_{2i-1}$ copies of the part $2i-1$ and $m_{2i}$ copies of $2i$ in $\pi$.  When this gets converted to $E_{m_{2i -1}+1}N_{m_{2i}+1}$ under $\psi$, it becomes $m_{2i}+1$ copies of a single part in $\psi(\pi)$.  Thus a part less than the largest part is repeated in $\psi(\pi)$ exactly when there is an even index $2i<\alpha$ for which $m_{2i}\geq 1$.  Moreover, if $\a$ is odd, then the mapping $$EN_{m_\a}N\mapsto E_{m_\a+1}N$$ guarantees that the largest part of $\psi(\pi)$ occurs once, while if $\a$ is even, then the mapping $$EN_{m_\a}N\mapsto N_{m_\a+1}N$$ guarantees that the largest part of $\psi(\pi)$ is repeated.  Thus the number of repeated parts in $\psi(\pi)$ is exactly $j$, so $\psi(\pi)\in \mathcal{FD}_{j,2}(n)$.

It remains to show that $\varphi$ and $\psi$ are indeed inverses.  We see that for $\pi \in \mathcal{FO}_{j,2}(n)$ with profile
\[
EN_{m_1}EN_{m_2} \cdots EN_{m_{(2\lceil\frac{\a}{2}\rceil -1)}}  \, \star,
\]
$\psi(\pi)$ has profile
\[
EE_{m_1}N_{m_2+1}EE_{m_3} N_{m_4+1}E \cdots E_{m_{(2\lceil\frac{\a}{2}\rceil -1)}}  \, \heartsuit,
\]
where 
\[
\heartsuit = \begin{cases} N &\text{if } \a \text{ odd}, \\  N_{m_\a +2} &\text{if } \a \text{ even}.  \end{cases}
\]
Thus applying $\varphi$ to $\psi(\pi)$ will map each $E_{m_{2i-1}}\mapsto N_{m_{2i-1}}$ and each $N_{m_{2i}+1}E \mapsto EN_{m_{2i}}E$.  Moreover, 
\[
\heartsuit \longmapsto \begin{cases} N &\text{if } \a \text{ odd}, \\  EN_{m_\a +1} &\text{if } \a \text{ even}. \end{cases}
\]
Thus $\varphi(\psi(\pi))$ has profile $EN_{m_1}EN_{m_2} \cdots EN_{m_{(2\lceil\frac{\a}{2}\rceil -1)}}  \, \star$, so $\varphi(\psi(\pi))=\pi$ as desired.

Showing the reverse direction requires more bookkeeping, but follows directly from the definitions after writing the profile of $\pi \in \mathcal{FD}_{j,2}(n)$ as
\[
E_{i_{1,0}}[NE_{i_{1,1}} \cdots NE_{i_{1,m_1}}](N_{r_1}) \cdots E_{i_{j,0}}[NE_{i_{j,1}} \cdots NE_{i_{j,m_j}}](N_{r_j}) [E_{i_{j+1,1}}N \cdots E_{i_{j+1,m_{j+1}}}N],
\]
where each $i_{a,b}\geq 1$, and $r_1, \ldots, r_j \geq 2$, and $m_1, \ldots, m_{j+1}\geq 0$ (i.e. each bracketed part may be empty), and grouping carefully when applying the maps.  
\end{proof}

\section{Proofs of Theorem \ref{perimST}}\label{sec:perimST}

For any set $X$ of positive integers, we define 
\begin{align*}
r_X(n) &= r(n \mid \text{parts in }X), \\
r_X(\a, \l, n) &= r_X(n \mid \text{largest part }\a \text{ and }\l \text{ parts}).
\end{align*}
Since $r_X(\a, \l, n)=0$ when $n\neq \a + \l -1$, we set $r_X(\a, n):=r_X(\a, n-\a +1, n)$ so that
\begin{equation}\label{lpsum}
r_X(n) = \sum_{\a=1}^n r_X(\a, n).
\end{equation}
Moreover, when $\a \in X$, we can count $r_X(\a, n)$ using the stars and bars counting technique, which gives that the number of ways to choose $r$ elements from a set of size $k$ allowing repetitions is $\binom{k+r-1}{r}$.  In this case, counting $r_X(\a, n)$ amounts to choosing the remaining $\l-1=n-\a$ parts from the set $X_{\a}=\{x\in X \mid x\leq \a \}$.  So when $\a \in X$,
\begin{equation}\label{lpcount}
r_X(\a, n) = \binom{|X_\a|+n-\a-1}{n-\a},
\end{equation}
and otherwise $r_X(\a, n)=0$.  

\begin{proof}[Proof of Theorem \ref{perimST} via counting functions]
Recall that $S$ and $T$ are ordered sets
\begin{align*}
S &=\{a_0, a_1, a_2, \dots\}, \\
T &= \{b_0, b_1, b_2, \dots\},
\end{align*}
where for all $i$, $a_{i} < a_{i+1}$, $b_{i} < b_{i+1}$, and $a_i \geq b_i$.  Define for any positive integer $n$, 
\begin{align*}
S_n & = \{a_i \in S: a_i \leq n\},\\
T_n & = \{b_i \in T: b_i \leq n\}.
\end{align*}
Then for a given $n$, $r_S(n) = r_{S_n}(n)$ and $r_T(n) = r_{T_n}(n)$.  Since $b_i \leq a_i$ for all $i$ it follows that $|S_n| \leq |T_n|$ for all $n$.  Furthermore, for any $i$, $|S_{a_i}|=|T_{b_i}|=i+1$. Thus from \eqref{lpsum} and \eqref{lpcount} we obtain when $n\geq a_0$, 
\begin{align}
r_S(n) & = \sum_{\a =1}^n r_S(\alpha, n) = \sum_{i=0}^{|S_n|-1} r_S(a_i, n) = \sum_{i=0}^{|S_n|-1} \binom{n-a_i+i}{n-a_i}, \label{rSsum} \\
r_T(n) & = \sum_{\a =1}^n r_T(\alpha, n) = \sum_{i=0}^{|T_n|-1} r_T(b_i, n) = \sum_{i=0}^{|T_n|-1} \binom{n-b_i+i}{n-b_i}, \label{rTsum}
\end{align}
and when $n<a_0$, $r_S(n)=0$, so $r_S(n)\leq r_T(n)$ holds trivially.  Since in general $\binom{n}{k} \leq \binom{n+m}{k+m}$ for integers $1\leq k \leq n$ and $m\geq 0$, and for all $i$ we have $a_i \geq b_i$, it follows that for all $i$,
\begin{align*}
\binom{n-a_i+i}{n-a_i} \leq  \binom{n-b_i+i}{n-b_i}.
\end{align*}
Moreover, since $|S_n| \leq |T_n|$, it follows that
\begin{align*}
r_S(n) = \sum_{i=0}^{|S_n|-1}r_S(a_i, n) \leq \sum_{i=0}^{|T_n|-1}r_T(b_i, n) = r_T(n).
\end{align*}    
\end{proof}

Let $R_X(n)$ be the set of partitions counted by $r_X(n)$ so that $|R_X(n)|=r_X(n)$.  We next give a proof of Theorem \ref{perimST} by constructing an injection $f : R_S(n) \rightarrow R_T(n)$.

\begin{proof}[Proof of Theorem \ref{perimST} via set injection]
Let $\pi \in R_S(n)$, and write 
\begin{equation}\label{piS}
\pi = a_0^{m_0} a_1^{m_1} \cdots a_k^{m_k +1},
\end{equation}
where $m_i \geq 0$ for each $i$, to denote that the multiplicity of the largest part $a_k$ is $m_k +1$, and the multiplicity of each smaller part $a_i<a_k$ in $\pi$ is $m_i\geq 0$.  Since $\pi\in R_S(n)$, the perimeter of $\pi$ is $$n=a_k + \sum_{i=0}^k m_i.$$  For $\pi$ in \eqref{piS} define $f(\pi)$ by
\[
f(\pi) = b_0^{m_0 + a_k - b_k} b_1^{m_1} \cdots b_k^{m_k +1}.
\]
Then the perimeter of $f(\pi)$ is again $n=a_k + \sum_{i=0}^k m_i$ and the parts are in $T$, so $f(\pi)\in R_T(n)$.  To see that $f$ is injective, suppose $\pi_1 = a_0^{m_0} a_1^{m_1} \cdots a_k^{m_k +1}$ and $\pi_2 = a_0^{n_0} a_1^{n_1} \cdots a_\ell^{n_\ell +1}$ such that
\[
f(\pi_1)=b_0^{m_0+a_k-b_k} b_1^{m_1} \cdots b_k^{m_k +1} = b_0^{n_0+a_\ell-b_\ell} b_1^{n_1} \cdots b_\ell^{n_\ell +1} = f(\pi_2).
\]
Then clearly $\ell=k$ and $n_i=m_i$ for all $1\leq i \leq k$ so it follows that $n_0=m_0$ and $\pi_1=\pi_2$.
\end{proof}

\section{Proof of Theorem \ref{beck}} \label{sec:Beck}

\begin{proof}
Since $f_d^{(a)}(\a,\l,n)=h_d^{(a)}(\a,\l,n)=0$ when $n\neq \a+\l-1$, let $f_d^{(a)}(\l,n) = f_d^{(a)}(n-\a+1,\l,n)$ and $h_d^{(a)}(\l,n) = h_d^{(a)}(n-\a+1,\l,n)$.  Chen et al. \cite{CHSS} give the following generating functions for $f_d^{(a)}(\l,n)$ and $h_d^{(a)}(\l,n)$,
\begin{align}
\sum_{n,\l \geq 1}^{\infty} f_d^{(a)}(\l,n)x^{n-\l+1}y^{\lambda}q^{n} &= \frac{x^a y q^a}{1 -y q-x^{d+1} q^{d+1}}, \label{fdafullgen}\\
\sum_{n,\l \geq 1} h_d^{(a)}(\l,n)x^{n-\l+1}y^{\lambda}q^{n} &= \frac{x^a y q^a}{1-x q-x^d y q^{d+1}}. \label{hdafullgen}
\end{align}
Setting $x=1$ in \eqref{fdafullgen} and \eqref{hdafullgen}, we have that
\begin{align}
\sum_{n,\l \geq 1}^{\infty} f_d^{(a)}(\l,n)y^{\lambda}q^{n} &= \frac{y q^a}{1 -y q- q^{d+1}}, \label{fdagen}\\
\sum_{n, \l \geq 1} h_d^{(a)}(\l,n)y^{\lambda}q^{n} &= \frac{y q^a}{1-q- y q^{d+1}}. \label{hdagen}
\end{align}
In general, if $G(w,q) = \sum_{n,m\geq 1} a(m,n)w^mq^n$, then 
\[
\frac{\partial}{\partial w} G(w,q) \Big\vert_{w=1} = \sum_{n,m \geq 1} m \cdot a(m,n) q^n. 
\]
So since $y$ is tracking the number of parts in partitions counted by $f_d^{(a)}(\l,n)$ and $h_d^{(a)}(\l,n)$ in \eqref{fdagen} and \eqref{hdagen}, respectively, it follows that $E(f_d^{(a)}(n),h_d^{(a)}(n))$ is the coefficient of $q^n$ in
\begin{align*}
&\frac{\partial}{\partial y}  \left( \frac{y q^a}{1 -y q-q^{d+1}} \right) \Big\vert_{y=1} - \frac{\partial}{\partial y}  \left( \frac{y q^a}{1-q- y q^{d+1}} \right) \Big\vert_{y=1} \\
&= \left( \frac{q^a(1-yq-q^{d+1}) + yq^{a+1}}{(1-yq-q^{d+1})^2} \right) \Big\vert_{y=1} - \left( \frac{q^a(1-q-yq^{d+1}) + yq^{a+1+d}}{(1-q-yq^{d+1})^2} \right) \Big\vert_{y=1} \\
&=\frac{q^{a+1} - q^{a+1+d}}{(1-q-q^{d+1})^2}.
\end{align*}
In other words,
\begin{equation}\label{Beckgen}
\sum_{n\geq 1} E(f_d^{(a)}(n),h_d^{(a)}(n)) q^n = \frac{q^{a+1} - q^{a+1+d}}{(1-q-q^{d+1})^2}.
\end{equation}
So it remains to show that the right hand side of \eqref{Beckgen} generates $fp_{d+1}^{(a,1)}(n) - fp_{d+1}^{(a,d+1)}(n)$.

Observe that setting $x=y=1$ in \eqref{fdafullgen} gives for any $1\leq a \leq d+1$, 
\[
\sum_{n \geq 1}^{\infty} f_d^{(a)}(n) q^{n} = \frac{q^a}{1 - q- q^{d+1}}.
\]
Thus, with \eqref{fpdab} we see that for any $1\leq a \leq d+1$, 
\begin{align*}
\sum_{n \geq 1}^{\infty}fp_{d+1}^{(a,1)}(n) q^{n} &= \left(\sum_{n \geq 1}^{\infty} f_d^{(a)}(n) q^{n} \right) \left(\sum_{n \geq 1}^{\infty} f_d^{(1)}(n) q^{n} \right) = \frac{q^{a+1}}{(1 - q- q^{d+1})^2}, \\
\sum_{n \geq 1}^{\infty}fp_{d+1}^{(a,d+1)}(n) q^{n} &= \left(\sum_{n \geq 1}^{\infty} f_d^{(a)}(n) q^{n} \right) \left(\sum_{n \geq 1}^{\infty} f_d^{(d+1)}(n) q^{n} \right) = \frac{q^{a+d+1}}{(1 - q- q^{d+1})^2},
\end{align*}
and so with \eqref{Beckgen} we have
\[
\sum_{n\geq 1} E(f_d^{(a)}(n),h_d^{(a)}(n)) q^n = \sum_{n \geq 1}^{\infty} (fp_{d+1}^{(a,1)}(n) - fp_{d+1}^{(a,d+1)}(n)) q^{n},
\]
as desired.
\end{proof}

\section{Proof of Proposition \ref{FPKKcase}} \label{sec:perim-kangkim}

We first establish a generating function for $\ell_d^{(a,b)}(n)$.  In general, for a partition $\pi$ with parts coming from $$S=\{c_0, c_1, c_2, \ldots,  \},$$ where each $c_i<c_{i+1}$, we write $$\pi = c_0^{m_0} c_1^{m_1} \cdots c_t^{m_t+1}$$ to denote that the multiplicity of the largest part $c_t$ is $m_t +1$, where $m_t \geq 0$ and the multiplicity of each smaller part $c_i<c_t$ in $\pi$ is $m_i\geq 0$.  From this perspective we see that the profile of $\pi$ can be written 
\begin{equation} \label{profile2}
E_{c_0}N_{m_0}E_{c_1-c_0}N_{m_1} \cdots E_{c_t-c_{t-1}}N_{m_t}N.
\end{equation} 

Let
\[
\ell_d^{(a,b)}(\a,\l, n) = r(n \mid \text{largest part } \a, \l \text{ total parts, parts in }S_{a,b}),
\]
where $S_{a,b} = \{x \in \N \mid x\equiv a \pmod{d} \text{ or } x\equiv b \pmod{d} \}$.  Thus in the notation above, $c_{2i} = a+di$ and $c_{2i+1}=b+di$ for any $i\geq 0$.  So translating \eqref{profile2}, we have that the profile of a partition $\pi$ counted by $\ell_d^{(a,b)}(\a,\l, n)$ has the form 
\begin{equation*}\label{profile3}
E_{a}N_{m_0}E_{b-a}N_{m_1}E_{a-b+d}N_{m_2}  \cdots E_{c_t-c_{t-1}}N_{m_t}N,
\end{equation*}
where $\a=c_t$ is the largest part of $\pi$.  We generate a choice of the first block $E_aN_{m_0}$ plus the final $N$ by 
\[
\frac{x^ayq^a}{1-yq}.
\]
To generate the terms $E_{b-a}N_{m_1} \cdots E_{a-b+d}N_{m_{2 \lfloor \frac{t}{2} \rfloor}}$, we observe that
\begin{align*}
f &= \frac{x^{b-a}q^{b-a}}{1-yq}, \\
g &= \frac{x^{a-b+d}q^{a-b+d}}{1-yq},
\end{align*}
generate one choice of block $E_{b-a}N_i$ or $E_{a-b+d}N_j$, respectively.  Thus, to generate equal numbers of terms $E_{b-a}N_i$ and $E_{a-b+d}N_j$ we use $1/(1-fg)$.  Lastly, to account for the possibility that $t$ is odd, we generate $E_{b-a}N_{m_t}$ (or nothing if $t$ even) by $(1+f)$.  Putting this together, we have that
\begin{equation*}\label{ell_dxyq}
\sum_{\a,\l,n\geq 1} \ell_d^{(a,b)}(\a,\l, n) x^\a y^\l q^n = \frac{x^ayq^a}{1-yq} \cdot \frac{1+f}{1-fg} = \frac{y(x^aq^a -x^ayq^{a+1} + x^bq^b) }{1-2yq+y^2q^2 -x^dq^d}.
\end{equation*} 

Plugging in $x=y=1$, gives that
\begin{equation}\label{ell_d}
\sum_{n\geq 1} \ell_d^{(a,b)}(n) q^n = \frac{q^a -q^{a+1} + q^b }{1-2q+q^2 -q^d}.
\end{equation} 

\begin{proof}[Proof of Proposition \ref{FPKKcase}]
We first note that the moreover statement follows from applying Theorem \ref{perimST}.

From \eqref{ell_d}, we have that
\begin{multline} \label{elldiff}
\sum_{n\geq 1} \left( \ell_d^{(a_1,b_1)}(n) - \ell_d^{(a_2,b_2)}(n) \right) q^n = \frac{(q^{b_1} - q^{b_2}) + (1-q)(q^{a_1} - q^{a_2})}{(1-q)^2 -q^d} \\
= \left(\frac{1}{1-\frac{q^d}{(1-q)^2}} \right) \left( q^{a_1}\left(\frac{1-q^{a_2-a_1}}{1-q}\right) +  \frac{q^{b_1}}{1-q}\left(\frac{1-q^{b_2-b_1}}{1-q}\right) \right).
\end{multline}
Since $a_1\leq a_2$ and $b_1 \leq b_2$, and they are not both equal, the (possibly finite) series 
\[
\sum_{n\geq 1} \nu(n) q^n = q^{a_1}\left(\frac{1-q^{a_2-a_1}}{1-q}\right) +  \frac{q^{b_1}}{1-q}\left(\frac{1-q^{b_2-b_1}}{1-q}\right)
\]
must satisfy $\nu(n)\geq 0$ for all $n\geq 1$ and there must exist $k\geq 1$ such that $\nu(k)\geq 1$ (for example $k=a_1$ works if $a_1\neq a_2$ and $k=b_1$ works if $b_1\neq b_2$).

Next, let
\[
\sum_{n\geq 0} \mu(n) q^n = \frac{1}{1-\frac{q^d}{(1-q)^2}}.
\]
Observe that since $f(q) = \frac{q^d}{(1-q)^2} = \sum_{n\geq 0} (n+1)q^{n+d}$ has positive integer coefficients that are strictly increasing, it follows from expanding $\frac{1}{1-f(q)}$ as a series in $f(q)$ that $\mu(n)$ also has positive integer coefficients that are strictly increasing.  Thus, from \eqref{elldiff} we have that 
\[
\sum_{n\geq 1} \left( \ell_d^{(a_1,b_1)}(n) - \ell_d^{(a_2,b_2)}(n) \right) q^n =  \left(\sum_{n\geq 0} \mu(n) q^n \right) \left( \sum_{n\geq 1} \nu(n) q^n \right) = \sum_{n\geq 1} \sum_{i=1}^n \mu(n-i)\nu(i) q^n.
\]
Since $\mu(n), \nu(n) \geq 0$ for all $n\geq 1$, and there exists $k\geq 1$ such that $\nu(k)\geq 1$, it follows that for $n\geq k$,
\[
\sum_{i=1}^n \mu(n-i)\nu(i) \geq  \mu(n-k).
\] 
Thus since $\mu(n)$ is strictly increasing in $n$, we have that 
\[
\lim_{n\rightarrow \infty} \left( \ell_d^{(a_1,b_1)}(n) - \ell_d^{(a_2,b_2)}(n) \right) = +\infty,
\]
as desired.
\end{proof}

We conclude by observing that \eqref{lpsum} and \eqref{lpcount} allow us to construct a counting function formula for $\ell_d^{(a,b)}(n)$ which could be of use in future studies.  For fixed positive integers $0 < a < b  \leq d$, let $X =  \{x \mid x \equiv a \! \pmod{d} \}$ and $Y =  \{y \mid y \equiv b \! \pmod{d} \}$, so that when $S=X\cup Y$ and $n \geq 1$,
\begin{equation} \label{ellsum}
\ell_d^{(a,b)}(n) = r_S(n) = \sum_{k=0}^{\lfloor \frac{n-a}{d} \rfloor} r_S(a+dk, n) +  \sum_{k=0}^{\lfloor \frac{n-b}{d} \rfloor} r_S(b+dk, n).
\end{equation}
Since $X$, $Y$ are disjoint, it follows that for any $1\leq \alpha\leq n$, 
\[
|S_\alpha| = |X_\alpha| + |Y_\alpha| =  \left \lfloor \frac{\a -a}{d} \right \rfloor + \left \lfloor \frac{\a -b}{d} \right \rfloor + 2.
\]    
Moreover, since $1\leq a < b \leq d$, it follows that $|\frac{\a-a}{d} - \frac{\a-b}{d}| <1$, and thus 
\[
|S_\a| = \begin{cases} 2( \frac{\a -a}{d})+1, \text{ when } \a\in X, \\  2( \frac{\a -b}{d})+2, \text{ when } \a\in Y. \end{cases}
\]
Thus, by \eqref{lpsum}, \eqref{lpcount}, and \eqref{ellsum}, it follows that
\begin{equation} \label{ellabcount}
\ell_d^{(a,b)}(n) = \sum_{k=0}^{\lfloor \frac{n-a}{d} \rfloor} \binom{2k + n-a-dk}{n-a-dk} +  \sum_{k=0}^{\lfloor \frac{n-b}{d} \rfloor} \binom{2k + 1 + n-b-dk}{n-b-dk}.
\end{equation}

\end{document}